\documentclass[11pt]{article}
\usepackage{amsthm}
\usepackage{amsfonts}
\usepackage{amsmath}
\usepackage{colonequals}
\usepackage{epsfig}
\usepackage{url}
\usepackage{hyperref}
\newcommand{\GG}{{\cal G}}

\newtheorem{theorem}{Theorem}
\newtheorem{corollary}[theorem]{Corollary}
\newtheorem{lemma}[theorem]{Lemma}

\theoremstyle{definition}
\newtheorem{definition}{Definition}

\title{On decidability of hyperbolicity}
\author{Zden\v{e}k Dvo\v{r}\'ak\thanks{Computer Science Institute, Charles University, Prague, Czech Republic. E-mail: 
\protect\href{mailto:rakdver@iuuk.mff.cuni.cz}{\protect\nolinkurl{rakdver@iuuk.mff.cuni.cz}}.
Supported by the ERC-CZ project LL2005 (Algorithms and complexity within and beyond bounded expansion) of the Ministry of Education of Czech Republic.}\and
 Luke Postle\thanks{Department of Combinatorics and Optimization, University of Waterloo, Waterloo, Ontario, Canada. E-mail:
 \protect\href{mailto:lpostle@uwaterloo.ca}{\protect\nolinkurl{lpostle@uwaterloo.ca}}. Canada Research Chair in Graph Theory.
 Partially supported by NSERC under Discovery Grant No. 2019-04304, the Ontario Early Researcher Awards program and the Canada
 Research Chairs program.}
}
\date{}

\begin{document}
\maketitle

\begin{abstract}
We prove that a wide range of coloring problems in graphs on surfaces can be resolved by inspecting a finite number of configurations.
\end{abstract}

\section{Introduction}

{\bf Reducible Configurations.} The method of reducible configurations\footnote{The basic idea of this method is to show that any graph $G$ from a considered class
of graphs contains a \emph{reducible configuration} $C$ (typically a subset of its vertices of bounded size
with prescribed adjacencies and degrees of vertices) with the property that every coloring of $G-C$ can be transformed
into a coloring of $G$.  One then proves that $G$ is colorable by a straightforward inductive argument.} is among the most important tools in the graph coloring
theory.
Indeed, the vast majority of coloring results in graphs drawn in the plane or other surfaces have been
proved using this method, including famously the Four Color Theorem~\cite{AppHak1,AppHakKoc,rsst},
as well as Gr\"otzsch' Theorem~\cite{grotzsch1959} and Borodin's acyclic coloring theorem~\cite{acyc}.
At the beginning of the 1990s, one would in fact be basically correct in stating that all known results
on coloring planar graphs can be proved using this method.

This changed with Thomassen's ingenious proof that all planar graphs are 5-choosable~\cite{thomassen1994}.
This central result in the chromatic theory of planar graphs has received a substantial
amount of attention, including several generalizations and strengthenings~\cite{thom-2007,LidDvoSkre,LidDvoMohPos,atar}.
Despite this, Thomassen's nibbling idea remains the only known way to prove it.  This led even to speculations that reducible
configurations for 5-choosability of planar graphs might not exist, making the result impossible to prove by
the reducible configurations method.

However, this turns out not to be the case.  Recently, Postle~\cite{postledistr} gave a quite general distributed coloring algorithm,
which in particular applies to 5-list-coloring of planar graphs.  The inspection of his argument shows that
every planar graph with $n$ vertices contains $\Omega(n)$ configurations reducible for $5$-choosability,
implying most of them have bounded size.  The argument uses the fact that graphs critical for $5$-choosability
form a hyperbolic family~\cite{lukethe} in the sense of the hyperbolicity theory of Postle and Thomas~\cite{PosThoHyperb};
we give the necessary definitions below, but for now it is sufficient to know this is a property of $5$-list-coloring
(as well as many other kinds of coloring) strengthening the fact that planar graphs are $5$-choosable.

In this note, we prove more is the case:  One can always establish hyperbolicity by examining a finite number of configurations.
Thus, although proving that planar graphs are colorable for a particular kind of coloring is often a very challenging problem,
if the critical graphs for this kind of coloring turn out to form a hyperbolic family, this proof can (at least in theory)
be obtained completely mechanically.  As a bonus, Postle and Thomas~\cite{PosThoHyperb} demonstrated the hyperbolicity has
a number of interesting structural and algorithmic consequences, which we thus gain for free.

\vskip.1in

\noindent{\bf Critical Graphs.} Let us now give the definitions necessary to make our statements precise.
We say that a graph $G$ is \emph{critical for $k$-coloring} if every proper subgraph of $G$ is $k$-colorable but $G$ itself is not.
Similarly, a graph $G$ is \emph{critical for $k$-choosability} if for some assignment $L$ of lists of size $k$ to vertices of $G$,
every proper subgraph of $G$ is $L$-colorable but $G$ itself is not.
The importance of critical graphs stems from the obvious observation that a graph is $k$-colorable if and only if it does
not contain any subgraph critical for $k$-coloring, and $k$-choosable if and only if it does not contain any subgraph
critical for $k$-choosability.

We mostly focus on graphs embedded in a fixed surface.  It is convenient to allow oneself to cut out parts of the surface and to consider
the subgraph drawn in the resulting surface with a boundary.  Throughout the paper, we implicitly allow the surfaces to have
a non-empty boundary, but we require that each graph $G$ drawn in a surface $\Sigma$ intersects the boundary only in vertices;
we let $\partial_G \Sigma$ denote the \emph{number} of vertices of $G$ contained in the boundary of $\Sigma$.
For a graph $G$ drawn in a surface with boundary $B$ and a list assignment $L$, we say $G$ is \emph{critical for $L$-coloring}
if for every proper subgraph $H$ of $G$ containing all vertices of $V(G)\cap B$, there exists an $L$-coloring of $V(G)\cap B$ that extends
to an $L$-coloring of $H$, but does not extend to a $k$-coloring of $G$.  We say $G$ is \emph{critical for $k$-choosability}
if there exists an assignment $L$ of lists of size $k$ to vertices of $G$ such that $G$ is critical for $L$-coloring,
and we say it is \emph{critical for $k$-coloring} if the same holds with $L$ giving each vertex the same list $\{1,\ldots,k\}$. 
The definition is motivated by the following observation:
Suppose $G$ is a graph embedded in a surface $\Sigma$ and $G$ is critical for $k$-choosability. If $\Sigma'$ is a subsurface of
$\Sigma$, the boundary of $\Sigma'$ intersects $G$ only in vertices, and $G'=G\cap\Sigma'$, then $G'$ is critical for $k$-choosability.

\vskip.1in

\noindent {\bf Hyperbolicity.} A key idea of hyperbolicity theory is that it suffices to focus on subgraphs drawn in subsurfaces homeomorphic to the disk or the cylinder;
more precisely, a slightly more general notion of surface containment which we are about to introduce is needed.
A \emph{face} of a graph $H$ drawn in a surface $\Sigma$ is a connected component of the space obtained from $\Sigma$ by deleting the points
of the drawing of $G$.
For another graph $G$ drawn in $\Sigma$, we say that the drawing of $H$ is \emph{$G$-normal}
if the drawings of $G$ and $H$ intersect exactly in $V(G)\cap V(H)$.  

\begin{definition} An open subset $\Delta\subseteq\Sigma$ is a \emph{$G$-slice}
if $\Delta$ is a face of some $G$-normal graph $H$ and the closure of $\Delta$ contains at least one vertex of $G$;
by $\partial_G \Delta$, we denote number of times the facial walks of the face $\Delta$ intersect $G$
(which may be larger than the number of vertices of $G$ contained in the boundary of $\Delta$, in case the walks pass through the same
vertex several times).  A $G$-slice is a \emph{$G$-disk} if it is homeomorphic to an open disk, and a \emph{$G$-cylinder} if it is
homeomorphic to an open cylinder.
\end{definition}

\vskip.1in

\noindent {\bf Hyperbolic Families of Graphs: Definitions and Examples.} For a real number $c>0$, we say that a graph $G$ drawn in a surface $\Sigma$ is \emph{$c$-hyperbolic}
if every $G$-disk $\Delta$ satisfies
\begin{equation}\label{eq-weak}
|V(G)\cap \Delta|\le c(\partial_G\Delta-1).
\end{equation}
Consider a class $\GG$ of graphs embedded in surfaces.  We say that $\GG$ is \emph{hyperbolic} if there exists a constant $c_\GG>0$
such that every graph in $\GG$ is $c_\GG$-hyperbolic.  In this case, we say that $c_\GG$ is a \emph{Cheeger constant} of $\GG$.
The class $\GG$ is \emph{strongly hyperbolic} if it is hyperbolic and furthermore, there exists a function $f_\GG:\mathbb{N}\to\mathbb{N}$
such that that for every $G\in \GG$, every $G$-cylinder\footnote{Note that here we deviate slightly from the notation of Postle and Thomas~\cite{PosThoHyperb},
who only constrain $G$-cylinders whose boundaries trace disjoint cycles in $G$; it is easy to see that the two definitions are (up to the choice
of the function $f_\GG$) equivalent, and the one we chose slightly simplifies some of the arguments.} $\Delta$ satisfies
$$|V(G)\cap \Delta|\le f_\GG(\partial_G\Delta).$$

In many important cases, the classes of critical graphs are strongly hyperbolic:
\begin{itemize}
\item For every $k\ge 5$, the class of all graphs drawn in a surface and critical for $k$-coloring is strongly hyperbolic~\cite{pothom}, even in the list coloring setting~\cite{lukethe}.
\item For every $k\ge 4$, the class of all triangle-free graphs drawn in a surface and critical for $k$-coloring is strongly hyperbolic even in the
list coloring setting, as shown by a straightforward density argument~\cite{galfor2}.
\item For every $k\ge 3$, the class of all graphs of girth at least five drawn in a surface and critical for $k$-coloring is strongly hyperbolic~\cite{trfree3}, even in the list coloring setting~\cite{postle3crit}.
\end{itemize}

\vskip.1in

\noindent {\bf Results for Hyperbolic Families.} Let us now summarize some of the results on strongly hyperbolic classes.
\begin{itemize}
\item Strong hyperbolicity implies that there are only finitely many graphs
from the class embedded in any surface without the boundary~\cite{PosThoHyperb}, or more generally, that the size of each graph $G$ in the class embedded in
a surface $\Sigma$ is bounded by a linear function of $\partial_G\Sigma$ and the genus of $\Sigma$.  In the coloring setting, this implies
that colorability or even extendability of a precoloring of a bounded number of vertices can be decided in polynomial time by
testing only finitely many obstructions.
\item In fact, a slightly more involved argument shows that only hyperbolicity is needed to
ensure the existence of polynomial-time algorithms~\cite{dvokawalg}.
\item Recently, Postle~\cite{postledistr} gave an efficient distributed coloring algorithm
assuming strong hyperbolicity.
\item Hyperbolicity is also sufficient to ensure that graphs in the class
embedded in a surface of genus $g$ without boundary have edgewidth $O(\log g)$; in the coloring setting, this ensures that all graphs
drawn with edgewidth $\Omega(\log g)$ are colorable~\cite{PosThoHyperb}.
\item If the class is strongly hyperbolic, one can furthermore enforce a precoloring
of any number of vertices that are sufficiently far apart~\cite{PosThoHyperb}.
\end{itemize}

Coming back to the proof perspective, establishing strong hyperbolicity of the considered class of graphs
(rather than say trying to directly show that there are only finitely many graphs in the class drawn in any fixed surface)
already simplifies the matters somewhat, since instead of considering all surfaces, one only needs to deal with graphs drawn in the disk or in the cylinder.
This still is a non-trivial task often requiring use of intricate techniques, typically either the method of reducible configurations or Thomassen's
nibbling method combined with detailed accounting.  However, building upon the ideas of~\cite{postledistr}, we prove that to establish
hyperbolicity or strong hyperbolicity with a given Cheeger constant, it actually suffices to inspect a finite number of graphs!

\vskip.1in

\noindent {\bf Main Results.} For a real number $c>0$ and a positive integer $t$, we say that the class $\GG$ is \emph{hyperbolic with Cheeger constant $c$
up to size $t$} if the condition (\ref{eq-weak}) holds for every $G\in \GG$ and every $G$-disk $\Delta$ such that $|V(G)\cap \Delta|+\partial_G\Delta\le t$.
Our first main result is that hyperbolicity up to sufficient (but bounded) size implies hyperbolicity (with a slightly worse
Cheeger constant).
\begin{theorem}\label{thm-weak}
Let $\GG$ be a class of graphs drawn on surfaces and let $c$ and $\varepsilon$ be positive real numbers. 
Let $t=\lceil 2b\log b\rceil$, where $b=\tfrac{2(c+1)(c+1+\varepsilon)(45c+45\varepsilon+81)}{\varepsilon}$.
If $\GG$ is hyperbolic with Cheeger constant $c$ up to size $t$,
then $\GG$ is hyperbolic with Cheeger constant $c+\varepsilon$.
\end{theorem}
Note that as a special case, hyperbolicity implies that there is no graph from the class embedded in the sphere.
Consequently, Theorem~\ref{thm-weak} implies that the fact that planar graphs are 5-choosable~\cite{thomassen1994} can be proven by inspecting a finite number of graphs. Similarly Theorem~\ref{thm-weak} implies the fact that planar graphs of girth at least five are $3$-choosable~\cite{thomassen1995-34} can be proven by inspecting a finite number of graphs.

For strong hyperbolicity of hyperbolic classes, we have the following analogous result, which is essentially implicit in~\cite{dvokawalg,PosThoHyperb}.
For a function $f:\mathbb{N}\to\mathbb{N}$ and a positive integer $t$, we say that the class $\GG$ is \emph{strongly $f$-hyperbolic up to size $t$}
if for every $G\in \GG$, every $G$-cylinder $\Delta$ such that $|V(G)\cap\Delta|+\partial_G\Delta\le t$ satisfies
$|V(G)\cap\Delta|\le f(\partial_G\Delta).$
\begin{theorem}\label{thm-strong}
Let $\GG$ be a hyperbolic class of graphs on surfaces with Cheeger constant $c$.  For a non-decreasing function
$f:\mathbb{N}\to\mathbb{N}$, let $t=f(\lceil 8c+4\rceil)+16c^2+24c+8$.  If $\GG$ is strongly $f$-hyperbolic up to size $t$,
then $\GG$ is strongly $g$-hyperbolic for the function $g:\mathbb{N}\to \mathbb{N}$ defined by
$g(k)=\lceil 2ck+8c^2+16c\rceil+6+f(\lceil 8c+4\rceil)$.
\end{theorem}

Before we proceed with the proofs, let us introduce another bit of a notation.
Suppose $G$ is a graph drawn in a surface and $\Delta$ is a $G$-slice. 
We would like to view the subgraph of $G$ drawn in the closure of $\Delta$ as a graph drawn in a surface;
to do so, we split in the natural way the vertices of $G$ that appear in the facial walks of $\Delta$
multiple times.
More precisely, note that there is a unique surface $\Pi$ whose interior is homeomorphic to $\Delta$.
Let $\theta:\Pi\to\overline{\Delta}$ be a continuous function such that the restriction of $\theta$ to the interior of $\Pi$
is a homeomorphism to $\Delta$.  For a $G$-slice $\Delta$, let $G_\Delta=f^{-1}(G\cap\overline{\Delta})$.
Let us remark that if $G$ is critical for $k$-coloring, then $G_\Delta$ drawn in $\Pi$ is also critical for $k$-coloring.
A $G$-slice is \emph{simple} if $\Pi$ is homeomorphic to $\overline{\Delta}$;
in this case, we implicitly take $\Pi=\overline{\Delta}$ and $f$ to be the identity function, and thus
$G_\Delta$ is just the subgraph of $G$ consisting of the vertices and edges drawn in $\overline{\Delta}$.

\section{Hyperbolicity}

In this section, we prove Theorem~\ref{thm-weak} on hyperbolic classes.
We need the following result, a variation on the well-known separator theorem of Lipton and Tarjan~\cite{lt79}.
We say that a graph $G$ drawn in a disk $\Sigma$ is \emph{internally $c$-hyperbolic} if the condition (\ref{eq-weak})
holds for all simple $G$-disks $\Delta\subset\Sigma$ such that $V(G_\Delta)\neq V(G)$.

\begin{lemma}\label{lemma-sep}
Let $G$ be a graph with $n$ vertices drawn in a disk $\Sigma$, and let $c>0$ be a real number.
If $G$ is internally $c$-hyperbolic, then there exists a simple $G$-disk $\Delta$ such that
$\tfrac{1}{3}n\le |G_\Delta|\le \tfrac{2}{3}n+2(c+1)\log n+4$ and
$$\frac{\partial_G\Delta}{|G_\Delta|}\le \frac{\partial_G\Sigma+6(c+1)\log n+9}{n}.$$
\end{lemma}
\begin{proof}
We say a face of $G$ is \emph{internal} if it does not intersect the boundary of $\Sigma$.
Without loss of generality, we may assume that $G$ is $2$-connected and all internal faces are triangles,
and the boundary of $\Sigma$ intersects $G$ exactly in vertices of a cycle $K$ such that all other vertices
and edges are contained in the open disk bounded by $K$. To accomplish this,
we suppress internal faces of length at most $2$ (bounded by loops or bigons) and then add edges
to connect $G$ and triangulate the internal faces.  

Let $G'$ be the plane triangulation obtained by pasting $\Sigma$ together with the drawing of $G$ in the plane
and adding a vertex $u$ adjacent to all vertices of $K$. 
For a positive integer $i$, let $n_i$ and $n^+_i$ denote the number of vertices of $G'$ at distance exactly $i$ and more than $i$ from $u$,
respectively.  We claim that for $i\ge 2$, we have $n^+_i\le \max(0, c(n_i-1))$.  Indeed, consider the graph $G_i$ obtained from $G'$ by deleting
all vertices at distance less than $i$ from $u$.  Note that for each $2$-connected block $B$ of $G_i$ with the outer face bounded by
a cycle $C_B$, there exists a simple $G$-disk $\Lambda\subseteq\Sigma$ with $G_\Lambda=B$ and containing exactly $V(C_B)$
in the boundary.  Since $G$ is
internally $c$-hyperbolic, we have $|B|-|C_B|\le c(|C_B|-1)$; and if $B_1$, \ldots, $B_m$ are the $2$-connected blocks of $G_i$ bounded by cycles,
we have $n^+_i=\sum_{i=1}^m (|B_i|-|C_{B_i}|)\le c\sum_{i=1}^m (|C_{B_i}|-1)\le c(n_i-1)$; the last inequality holds unless $n_i=0$,
in which case we have $n^+_i=0$ since $G$ is connected.

Therefore, for $i\ge 2$, if $n_i\neq 0$, then $n_i>n^+_i/c$, and thus
$n^+_{i-1}=n^+_i+n_i>(1+1/c)n^+_i$.  Hence, if $n^+_{i-1}\neq 0$, then
$n^+_i<\bigl(1-\tfrac{1}{c+1}\bigr)n^+_{i-1}$.
Let $d=\lceil (c+1)\log n\rceil+1$.
Since $n^+_{1}<n$, we have $n^+_{d}<\bigl(1-\tfrac{1}{c+1}\bigr)^{d-1}n\le 1$, and thus every vertex of $G'$ is at distance at most $d$ from $u$.

Let $T$ be a BFS spanning tree of $G'$ rooted in $u$.  Note that the dual of $G'$ has a spanning tree $T^\star$ whose edges cross
exactly the edges of $G'$ not belonging to $T$.  Consider any edge $e\in E(G')\setminus E(T)$, and let $C_e$
denote the unique cycle in $T+e$.  The graph $T+e$ has two faces, let $f_e$ denote the one incident with fewer vertices of $G$
and let $A_e$ denote the set of these vertices, let $B_e$ denote the set of vertices incident with the other face of $T_e$,
and direct the corresponding edge $e^\star$ of $T^\star$ into $f_e$.  Note that $(A_e,B_e)$ is a separation of $G$ and $A_e\cap B_e=V(C_e)\setminus\{u\}$.
The directed tree $T^\star$ contains a source, corresponding to a face $f$ of $G'$.  Letting $F$ denote the set of edges of
$E(G')\setminus E(T)$ incident with $f$, we have $V(G)=\bigcup_{e\in F} A_e$. Since $G'$ is a triangulation, we have $|F|\le 3$,
and thus there exists an edge $e\in F\subseteq E(G')\setminus E(T)$ such that $|A_e|\ge \tfrac{n}{3}$; let us fix such an edge $e$.
Note that $|B_e|\ge |A_e|\ge\tfrac{n}{3}$.

Suppose that $u\in V(C_e)$; then $C_e-u$ is a path in $G$ with both ends in $K$ and otherwise disjoint from $K$,
and $C_e-u$ has at most $2d$ vertices.  Let $\Delta_A$ and $\Delta_B$ be simple $G$-disks in $\Sigma$
tracing the cycles in $K+(C_e-u)$, such that $A_e=V(G_{\Delta_A})$ and $B_e=V(G_{\Delta_B})$.
Note that
$$\partial_G\Delta_A+\partial_G\Delta_B\le \partial_G\Sigma+4d\le\frac{\partial_G\Sigma+4d}{n}(|A_e|+|B_e|).$$
By symmetry, we may assume without loss of generality that $\partial_G\Delta_A\le \frac{\partial_G\Sigma+4d}{n}|A_e|$.
Now let $\Delta=\Delta_A$. We have $|G_\Delta|=|A_e|\ge\tfrac{1}{3}n$, $|G_\Delta|\le n-|B_e|+|V(C_e-u)|\le \tfrac{2}{3}n+2d$, and
$$\frac{\partial_G\Delta}{|G_\Delta|}=\frac{\partial_G\Delta_A}{|A_e|}\le \frac{\partial_G\Sigma+4d}{n}.$$

Similarly, if $u\not\in V(C_e)$, then $|C_e|\le 2d-1$ and $C_e$ is a cycle in $G$,
and we may let $\Delta$ be a simple $G$-disk in $\Sigma$ tracing $C_e$.  In this case,
$$\frac{\partial_G\Delta}{|G_\Delta|}\le \frac{2d-1}{n/3}\le \frac{\partial_G\Sigma+6d-3}{n}.$$
\end{proof}

We now iterate Lemma~\ref{lemma-sep} in order to further decrease the size of the graph.

\begin{corollary}\label{cor-sep}
Let $G$ be a graph with $n$ vertices embedded in a disk $\Sigma$, and let $c>0$ and $t$ be real numbers
such that $t\ge 24(c+1)\log t + 48$.
If $G$ is internally $c$-hyperbolic, then there exists a simple $G$-disk $\Delta$ such that
$|G_\Delta|\le t$ and
$$\frac{\partial_G\Delta}{|G_\Delta|}\le \frac{\partial_G\Sigma}{n}+\frac{21c+57+24(c+1)\log t}{t}.$$
\end{corollary}
\begin{proof}
The claim is trivial if $n\le t$, as then we can select $\Delta=\Sigma$.  Hence, suppose that $n>t$.
We repeatedly apply Lemma~\ref{lemma-sep} until the graph $G_\Delta$ we obtain has at most $t$ vertices.
Each time, when applied to a graph of size $k>t$, the number of the vertices of the considered graph decreases by
a factor of at most
$$\frac{2}{3}+\frac{2(c+1)\log k+4}{k}\le \frac{2}{3}+\frac{2(c+1)\log t+4}{t}\le \frac{3}{4}.$$
Thus, for each integer $m\ge 0$, Lemma~\ref{lemma-sep} is applied during the process to at most one graph with 
at least $(4/3)^mt$ but less than $(4/3)^{m+1}t$ vertices.  Consequently,
\begin{align*}
\frac{\partial_G\Delta}{|G_\Delta|}-\frac{\partial_G\Sigma}{n}&\le \sum_{m\ge 0} \frac{6(c+1)\log ((4/3)^mt)+9}{(4/3)^mt}\\
&=\frac{6(c+1)\log(4/3)}{t}\sum_{m\ge 0} m(3/4)^m+\frac{6(c+1)\log t+9}{t}\sum_{m\ge 0} (3/4)^m.
\end{align*}
Since $\sum_{m\ge 0} m(3/4)^m = 12$ and $\sum_{m\ge 0} (3/4)^m=4$, we have that
\begin{align*}
\frac{\partial_G\Delta}{|G_\Delta|}-\frac{\partial_G\Sigma}{n}&\le \frac{6(c+1)\log(4/3)}{t}\cdot 12+\frac{6(c+1)\log t+9}{t}\cdot 4\\
&\le\frac{21c+57+24(c+1)\log t}{t}.
\end{align*}
\end{proof}

We are now ready to prove our first main result.
\begin{proof}[Proof of Theorem~\ref{thm-weak}]
Let $c'=c+\varepsilon$ and $\beta=\tfrac{\varepsilon}{2(c+1)(c'+1)}$.
The choice of $t$ implies
$$\frac{21c'+57+24(c'+1)\log t}{t}\le \frac{(45c'+81)\log t}{t}\le \frac{45c'+81}{b}=\beta;$$
since $\beta<\tfrac{1}{2(c+1)}<1/2$, this also ensures $t\ge 24(c'+1)\log t + 48$.

Suppose for a contradiction that $\GG$ is not
hyperbolic with Cheeger constant $c'$.  Hence, there exists $G\in \GG$ and a $G$-disk $\Sigma$ such that
$|V(G)\cap\Sigma|>c'(\partial_G\Sigma-1)$.  Choose such a graph $G$ and a disk $\Sigma$ with $|G_\Sigma|$ minimum.
Note that $|G_\Sigma|=|V(G)\cap\Sigma|+\partial_G\Sigma>(c'+1)\partial_G\Sigma-c'$.
Because $\GG$ is hyperbolic up to size t, we have that $|G_\Sigma|>t$, and thus
$$\frac{\partial_G\Sigma}{|G_\Sigma|}<\frac{(|G_\Sigma|+c')/(c'+1)}{|G_\Sigma|}<\frac{1}{c'+1}+\frac{1}{|G_\Sigma|}<\frac{1}{c'}+\frac{1}{t}.$$
The minimality of $|G_\Sigma|$ implies that $G_\Sigma$ is internally $c'$-hyperbolic.  By Corollary~\ref{cor-sep},
there exists a simple $G_\Sigma$-disk $\Delta$ such that $|G_\Delta|\le t$ and
$$\frac{\partial_G\Delta}{|G_\Delta|}\le \frac{\partial_G\Sigma}{|G_\Sigma|}+\beta< \frac{1}{c'+1}+\frac{1}{t}+\beta<\frac{1}{c'+1}+2\beta=\frac{1}{c+1}.$$
Now, we view $\Delta$ as a $G$-disk rather than a $G_\Sigma$-disk.  But then we have
$$|G\cap\Delta|=|G_\Delta|-\partial_G\Delta>(c+1)\partial_G\Delta-\partial_G\Delta>c(\partial_G\Delta-1),$$
contradicting the assumption that $\GG$ is hyperbolic with Cheeger constant $c$ up to size $t$.
\end{proof}

\section{Strong hyperbolicity}

Suppose $\Delta$ is a $G$-cylinder and let $S_1$ and $S_2$
be the sets of vertices of $G$ contained in the two components of the boundary of the cylinder $\Pi$ in which $G_\Delta$ is embedded.
Let $d$ be the distance between $S_1$ and $S_2$ in $G_\Delta$.
For a real number $a>0$, we say that $\Delta$ is \emph{$a$-fat} if $S_1\neq \emptyset\neq S_2$, the distance $d$ between $S_1$ and $S_2$
in $G$ is finite, and for $i=1,\ldots,d-1$, $G_{\Delta}$ contains at least $a$ vertices at distance exactly $i$ from $S_1$.
To prove Theorem~\ref{thm-strong}, we need the following observation on fat cylinders.
\begin{lemma}\label{lemma-fat}
Suppose a graph $G$ embedded in a surface is $c$-hyperbolic for some real number $c>0$.
Then every $(4c+2)$-fat $G$-cylinder $\Delta$ satisfies $|V(G)\cap \Delta|<2c\partial_G\Delta+4c+2$.
\end{lemma}
\begin{proof}
Without loss of generality, we may assume that $G$ is embedded in a cylinder $\Pi$ and $\Delta$ is the interior of $\Pi$.
Let $S_1$ and $S_2$ be the sets of vertices of $G$ contained in the two components of the boundary of $\Pi$, and let $P$ be a shortest
path from $S_1$ to $S_2$ in $G$, of length $d$.  Note that there exists a $G$-disk $\Lambda$ whose boundary traces the boundary of $\Pi$
and passes twice along $P$; we have $\partial_G\Lambda = \partial_G\Delta + 2d$.  Since $G$ is $c$-hyperbolic,
we have $|V(G)\cap\Delta|-(d-1)=|V(G)\cap \Lambda|<c\partial_G\Lambda$, and thus
$|V(G)\cap \Delta|<c\partial_G\Delta+(2c+1)d$.  On the other hand, since $\Delta$ is $(4c+2)$-fat,
we have $|V(G)\cap \Delta|\ge 2(2c+1)(d-1)$, and thus $(2c+1)d\le |V(G)\cap \Delta|/2+2c+1$.
Combining the inequalities, we conclude that $|V(G)\cap\Delta|<2c\partial_G\Delta+4c+2$.
\end{proof}

The proof of the strong hyperbolicity result is now straightforward.

\begin{proof}[Proof of Theorem~\ref{thm-strong}]
Consider a $G$-cylinder $\Delta$ for a graph $G\in \GG$; we need to argue that $|V(G)\cap \Delta|\le g(\partial_G\Delta)$.
Without loss of generality, we may assume that $G$ is embedded in a cylinder $\Pi$ and $\Delta$ is the interior of $\Pi$, and thus $G_\Delta=G$.
Let $S_1$ and $S_2$ be the sets of vertices of $G$ contained in the two components of the boundary of $\Pi$; for $i\in \{1,2\}$, let $e_i=1$
if $S_i$ is non-empty and $e_i=0$ otherwise.  We will actually show that $|V(G)\cap \Delta|\le g(\partial_G\Delta-e_1-e_2)-e_1-e_2$.  Thus, we may without loss
of generality shift the boundary of $\Pi$ to ensure that $S_1\neq \emptyset\neq S_2$.  Furthermore, we can add edges to $G$ to ensure it is
connected.  Let $d$ be the distance between $S_1$ and $S_2$ in $G$.  Let $I\subseteq\{1,\ldots,d-1\}$ be the set of distances $i$
such that $G$ contains fewer than $4c+2$ vertices at distance exactly $i$ from $S_1$.  Let $i_1<\ldots<i_k$ be the elements
of $I$ in increasing order.  For $j=1,\ldots,k$, note that there exists a simple closed curve $\gamma_j$
intersecting $G$ only in vertices at distance exactly $i_j$ from $S_1$ and separating $S_1$ from $S_2$.  Let $\gamma_0$ and $\gamma_{k+1}$ be the curves
tracing the boundaries of $\Pi$ containing $S_1$ and $S_2$, respectively.

For $0\le i\le j\le k+1$, let $\Delta_{i,j}$ be the $G$-cylinder
bounded by $\gamma_i$ and $\gamma_j$.  For $j=0,\ldots,k$, note that $\Delta_{j,j+1}$ is $(4c+2)$-fat,
and thus $|V(G)\cap \Delta_{j,j+1}|\le 2c\partial_G\Delta_{j,j+1}+4c+2$ by Lemma~\ref{lemma-fat}.
Consequently, if $1\le j\le k-1$, then $|V(G)\cap \Delta_{j,j+1}|\le 16c^2+12c+2$.
Note that $|V(G)\cap \Delta_{1,k}|\le f(\lceil 8c+4\rceil)$; indeed, otherwise there exists the minimum index $j\le k$ such that
$$|V(G)\cap \Delta_{1,j}|>f(\lceil 8c+4\rceil)\ge f(\partial_G \Delta_{1,j}),$$ and then by the minimality of $j$,
we have that $$|G_{\Delta_{1,j}}|\le |V(G)\cap\Delta_{1,j-1}|+|V(G)\cap \Delta_{j-1,j}|+12c+6\le f(\lceil 8c+4\rceil)+16c^2+24c+8=t$$
contradicting the assumption that $\GG$ is strongly $f$-hyperbolic up to size~$t$.

It follows that
\begin{align*}
|V(G)\cap \Delta|&=|V(G)\cap \Delta_{0,1}|+|G_{\Delta_{1,k}}|+|V(G)\cap \Delta_{k,k+1}|\\
&\le 2c(\partial_G\Delta_{0,1}+\partial_G\Delta_{k,k+1})+8c+4+f(\lceil 8c+4\rceil)\\
&\le 2c\partial_G\Delta+8c^2+16c+4+f(\lceil 8c+4\rceil)\le g(\partial_G\Delta-2)-2,
\end{align*}
as required.
\end{proof}

\section{Further Remarks: Practical Considerations}

Let us discuss some practical and computational issues.  Firstly, Theorem~\ref{thm-weak} only enables us
to verify that the class is hyperbolic for a particular Cheeger constant, but does not tell us how to find such a constant.
For a class which is actually hyperbolic, we can in principle test larger and larger constants and we are guaranteed to eventually
succeed, but this does not give us a way to decide the class actually is not hyperbolic.  This likely cannot be avoided, since there
are hyperbolic classes with arbitrarily large Cheeger constants.

On the other hand, in practice one could likely get a good idea of what Cheeger constant to aim for by experimentation with small
graphs, or to discover a way to construct counterexamples to hyperbolicity in the process.  This brings us to the issue of how
large a Cheeger constant one can expect.  While earlier proofs often came up with quite large constants (around 2000 in~\cite{trfree3}),
the later works significantly improve on these bounds (less than 13 in~\cite{dk}, 19 in~\cite{list5col}, albeit with slightly
different definitions), and empirical evidence suggests
that the actual best possible values are even smaller.  A detailed inspection of the proof of Theorem~\ref{thm-weak} shows that
to establish that a class with Cheeger constant 5 is hyperbolic with Cheeger constant 10 (i.e., the case $c=\varepsilon=5$),
one would need to prove 5-hyperbolicity for graphs with less than $90\,000$ vertices.  As the number of graphs in a class that have a given
number of vertices typically grows exponentially, this is wildly beyond the reach of conceivable computational power.  Although
with a better analysis the bounds could likely be substantially improved, we have to concede our result is mostly of a theoretical interest.

Of course, it is plausible that actually significantly smaller reducible configurations exist.  However, to date no one has been able to
come up with an explicit list of reducible configurations for 5-choosability of planar graphs yet, indicating that the graphs in such a list likely cannot be very small.

Finally, let us remark on what one would actually have to do to establish hyperbolicity of a class $\GG$ up to size $t$.
For each graph $H$ with at most $t$ vertices drawn in a disk $\Sigma$ such that $|H|-\partial_H\Sigma>c(\partial_H\Sigma-1)$, one has to verify that there
is no graph $G\in\GG$ drawn in any surface and a $G$-disk $\Delta$ such that $H$ is homeomorphic to $G_\Delta$.
In the case that $\GG$ is the class of graphs
on surfaces of girth at least $g$ and critical for $k$-coloring, this amounts to showing that either $H$ contains a cycle of length less than $g$,
or that $H$ is not critical for $k$-coloring.   Short cycles can be detected in polynomial time, while to verify the criticality one likely
has to use an exponential-time brute-force algorithm; nevertheless, for reasonably small graphs both tasks are tractable.
This does not necessarily have to be the case for general hyperbolic classes, although for the ones coming from the coloring
problems, it is likely to again work out to be equivalent to verifying a suitable version of criticality with precolored boundary.

\bibliographystyle{siam}
\bibliography{../data}

\end{document}